\documentclass{article}

\usepackage{amssymb}
\usepackage{amsmath}
\usepackage{stmaryrd}
\usepackage{float}
\usepackage{graphicx}
\usepackage{tikz}
\usepackage{cite}

\usetikzlibrary{shapes.geometric, fit, positioning, arrows}
\usepackage{hyperref}

\newtheorem{assumption}{Assumption}

\usepackage[colors]{optsys}

\renewcommand{\va}[1]{{\boldsymbol{\uppercase{#1}}}}
\renewcommand{\opt}{{\sharp}}

\newcommand{\price}{p}
\newcommand{\Price}{\va{p}}
\newcommand{\PRICE}{\mathcal{P}}

\newcommand{\alloc}{r}
\newcommand{\Alloc}{\va{r}}
\newcommand{\ALLOCATION}{\mathcal{R}}

\newcommand{\cZ}{\mathcal{Z}}
\newcommand{\cF}{\mathcal{F}}
\newcommand{\cG}{\mathcal{G}}

\newcommand{\st}{\text{s.t.}}

\newcommand{\dynamic}{{g}}

\newcommand{\x}{\boldsymbol{X}}
\renewcommand{\u}{\boldsymbol{U}}
\newcommand{\w}{\boldsymbol{W}}

\newcommand{\pvaluefunc}{\underline{V}}
\newcommand{\qvaluefunc}{\overline{V}}

\renewcommand{\horizon}{T}
\newcommand{\final}{T}
\renewcommand{\feedback}{\gamma}                              
\renewcommand{\policy}{\feedback}                             

\newcommand{\CONE}{S}
\newcommand{\CONESTO}{\boldsymbol{\mathcal{S}}}

\newcommand{\pscal}[2]{#1 \cdot #2}
\renewcommand{\pscal}[2]{\big\langle#1\:,#2\big\rangle}     

\let\citep=\cite
\newcommand{\card}[1]{\vert #1 \vert}
\newcommand{\NODES}{\mathfrak{N}} 
\newcommand{\ARCS}{\mathfrak{A}} 
\newcommand{\site}{\node}
\newcommand{\SITE}{\NODES}

\renewcommand{\defegal}{=}

\newcommand{\FORALLTIMES}[3]{\forall #1\in\ic{#2,#3}}

\newcommand{\post}{{t+1}}
\newcommand{\LL}{{\mathbb L}}

\usepackage{fullpage}
\graphicspath{{./}}

\title{Mixed Spatial and Temporal Decompositions
  for Large Scale Multistage Stochastic
  Optimization Problems}

\author{Pierre~Carpentier\thanks{UMA, ENSTA Paris, IP Paris, France}
  \and
  Jean-Philippe~Chancelier\thanks{CERMICS, Ecole des Ponts, Marne-la-Vall\'ee, France}
  \and
  Michel~De~Lara\footnotemark[2]
  \and
  Fran\c{c}ois~Pacaud\footnotemark[2]
}

\date{July 2020}

\begin{document}
\maketitle

\begin{abstract}
  We consider multistage stochastic optimization problems
  involving multiple units. Each unit is a (small) control system.
  Static constraints couple units at each stage.
  We present a mix of spatial and temporal decompositions
  to tackle such large scale problems.
  More precisely, we obtain theoretical bounds and policies
  by means of two methods, depending whether the coupling constraints
  are handled by prices or by resources.
  We study both centralized and decentralized information structures.
  We report the results of numerical experiments on the
  management of urban microgrids. It appears that decomposition methods are
  much faster and give better results than the standard Stochastic Dual
  Dynamic Programming method, both in terms of bounds and of policy performance.
\end{abstract}



\section{Introduction}

Multistage stochastic optimization problems are, by essence,
complex because their solutions are indexed both by stages
(time) and by uncertainties (scenarios).
Another layer of complexity can come from spatial structure.
The large scale nature of such problems makes decomposition methods
appealing (we refer to
\citep{ruszczynski1997decomposition,carpentier2017decomposition}
for a broad description of decomposition methods in stochastic optimization problems).

We sketch decomposition methods along three dimensions:
\emph{temporal decomposition} methods like Dynamic Programming
break the multistage problem into a sequence of interconnected static
subproblems \citep{bellman57,bertsekas1995dynamic};
\emph{scenario decomposition} methods split large scale stochastic
optimization problems scenario by scenario, yielding deterministic
subproblems \citep{rockafellar1991scenarios,watson2011progressive,kim2018algorithmic};
\emph{spatial decomposition} methods break possible spatial
couplings in a global problem to obtain local decoupled subproblems
\cite{cohen80}.
These decomposition schemes have been applied in many fields,
and especially in energy management: Dynamic Programming methods have been
used for example in dam management \citep{shapiro2012final},
and scenario decomposition has  been successfully applied to
the resolution of unit-commitment problems \citep{bacaud2001bundle},
among others.

Recent developments have mixed spatial decomposition
methods with Dynamic Programming to solve
large scale multistage stochastic optimization problems.
This work led to the introduction of the Dual Approximate
Dynamic Programming (DADP) algorithm, which was first
applied to unit-commitment problems with a single central
coupling constraint linking different stocks
\citep{barty2010decomposition}, and later applied to dams
management problems~\citep{carpentier2018stochastic}.
This article moves
one step further by considering altogether two types
of decompositions (by prices and by resources)
when dealing with general coupling
constraints among units.
General coupling constraints often arise from flows
conservation on a graph, 
and our motivation indeed comes from district microgrid
management, where buildings (units) consume,
produce and store energy and are interconnected through
a network.

The paper is organized as follows.
In Sect.~\ref{chap:nodal}, we introduce a generic stochastic
multistage problem with different subsystems linked together
via a set of static coupling constraints.
We present price and resource decomposition schemes, that make use of
so-called admissible coordination processes. We show how to
bound the global Bellman functions above by a sum of local
resource-decomposed value functions, and below by a sum of
local price-decomposed value functions.
In Sect.~\ref{sec:genericdecomposition}, we study the special
case of deterministic coordination processes. First, we show
that the local price and resource decomposed value functions
satisfy recursive Dynamic Programming equations. Second, we
outline how to improve the bounds. Third, we show how to use the decomposed Bellman
functions to devise admissible policies for the global problem.
Finally, we provide an analysis of the decentralized information
structure, that is, when the controls of a given subsystem only
depend on the past observations of the noise in that same subsystem.
In Sect.~\ref{chap:district:numerics}, we present numerical
results for the optimal management of different microgrids
of increasing size and complexity.
We compare the two decomposition algorithms with
(state of the art) Stochastic Dual Dynamic Programming (SDDP) algorithm.
The analysis of case studies consisting of district
microgrids coupling up to 48 buildings together enlightens
that decomposition methods give better results in terms of cost
performance, and achieve up to a four times speedup in terms of computational
time.

\section{Upper and Lower Bounds by Spatial Decomposition}
\label{chap:nodal}

We focus in \S\ref{sec:nodal:generic} on a generic decomposable
optimization problem and present price and resource decomposition
schemes. In~\S\ref{sec:nodal:globalpb}, we apply these two
methods to a multistage stochastic optimization problem, by decomposing
a global static coupling constraint by means of so-called price
and resource coordination processes. For such problems,
we define the notions of centralized and decentralized information structures.

\subsection{Bounds for an Optimization Problem under Coupling Constraints
  via Decomposition}
\label{sec:nodal:generic}

In~\S\ref{sec:nodal:genericproblem},
we introduce a generic optimization problem with coupled local units.
In~\S\ref{subsec:nodal:bounds}, we show how to bound its optimal value by
decomposition.

\subsubsection{Global Optimization Problem Formulation}
\label{sec:nodal:genericproblem}

Let~$\SITE$ be a finite set, representing local units~\( \site \in \SITE \)
(we use the letter~$\SITE$ as units can be seen as nodes on a graph).
Let~$\sequence{\cZ^\site}{\site \in \SITE}$ be a family of sets and
$J^\site: \cZ^\site \rightarrow \OpenIntervalClosed{-\infty}{+\infty}$, $\site \in \SITE$,
be local criteria, one for each unit, taking values in the extended reals
\( \OpenIntervalClosed{-\infty}{+\infty} \) ($+\infty$ included to allow for possible constraints).
Let $\sequence{\ALLOCATION^\site}{\site \in \SITE}$, be a family of vector spaces
and $\vartheta^\site: \cZ^\site \rightarrow \ALLOCATION^\site$, $\site \in \SITE$,
be mappings that model local constraints.

From these \emph{local} data, we formulate a \emph{global}
minimization problem under constraints. We define the product
set~$\cZ = \prod_{\site\in \SITE} \cZ^\site $ and the product
space~$\ALLOCATION=\prod_{\site\in \SITE} \ALLOCATION^\site$.
Finally, we introduce a subset $\CONE \subset  \ALLOCATION$
that captures the coupling constraints between the $N$~units.
Using the notation \( z=\sequence{z^\site}{\site \in \SITE} \),
we define the \emph{global optimization} problem as
\begin{subequations}
  \label{eq:gen:genpb}
  \begin{equation}
    V\opt = \inf_{z \in \cZ} \;
    \sum_{\site \in \SITE} J^\site(z^\site) \eqfinv
  \end{equation}
  under the \emph{global coupling constraint}
  \begin{equation}
    \label{eq:gen:coupling}
    \ba{\vartheta^\site(z^\site)}_{\site \in \SITE} \in -\CONE \eqfinp
  \end{equation}
\end{subequations}
The set~$\CONE$ is called the \emph{primal admissible set},
and an element~$\sequence{\alloc^\site}{\site \in \SITE}\in-\CONE$
is called an \emph{admissible resource vector}.
We note that, without Constraint~\eqref{eq:gen:coupling},
Problem~\eqref{eq:gen:genpb} would decompose into $|\SITE|$ independent
subproblems in a straightforward manner.

We moreover assume that, for $\site \in \SITE$, the space~$\ALLOCATION^\site$
(resources) is paired with a space~$\PRICE^\site$ (prices) by bilinear forms
$\bscal{\cdot}{\cdot}\; : \; \PRICE^\site \times \ALLOCATION^\site \to \RR$
(duality pairings).
We define the product space~$\PRICE= \prod_{\site \in \SITE} \PRICE^\site$,
so that~$\ALLOCATION$ and~$\PRICE$ are paired by the duality pairing
$\bscal{\price}{\alloc} = \sum_{\site\in \SITE} \bscal{\price^\site}{\alloc^\site}$
(see \cite{rockafellar1974conjugate} for further details; a typical
example of paired spaces is a Hilbert space and itself).

\subsubsection{Upper and Lower Bounds from Price and Resource Value Functions}
\label{subsec:nodal:bounds}

Consider the global optimization problem~\eqref{eq:gen:genpb}.
For each~$\site \in \SITE$, we introduce \emph{local price value
  functions} $\underline V^\site : \PRICE^\site \to \ClosedIntervalOpen{-\infty}{+\infty}$
defined by
\begin{equation}
  \label{eq:nodal:genericdualvf}
  \underline V^\site\nc{\price^\site}
  =  \inf_{z^\site \in \cZ^\site} \; J^\site(z^\site) +
  \bscal{\price^\site}{\vartheta^\site(z^\site)}
  \eqfinv
\end{equation}
where we have supposed that \( V^\site\nc{\price^\site} < +\infty \),
and \emph{local resource value functions}
$\overline V^\site: \ALLOCATION^\site \to  \OpenIntervalClosed{-\infty}{+\infty}$
defined by
\begin{equation}
  \label{eq:nodal:genericprimalvf}
  \overline V^\site\nc{\alloc^\site}
  = \inf_{z^\site\in \cZ^\site} \; J^\site(z^\site)
  \quad \st\ \quad \vartheta^\site(z^\site) = \alloc^\site
  \eqfinv
\end{equation}
where we have supposed that \( \overline V^\site\nc{\alloc^\site} > -\infty \),

We denote by $\CONE^\star \subset \PRICE$ the dual cone
associated with the constraint set~$\CONE$:
\begin{equation}
  \label{eq:nodal:dualcone}
  \CONE^\star \defegal
  \ba{\price \in \PRICE \; |\;
    \bscal{\price}{\alloc} \geq 0 \eqsepv
    \forall \alloc \in \CONE} \eqfinp
\end{equation}
The cone~$\CONE^\star$ is called the \emph{dual admissible set},
and an element~$\sequence{\price^\site}{\site \in \SITE}\in\CONE^\star$
is called an \emph{admissible price vector}.
We now establish lower and upper bounds for Problem~\eqref{eq:gen:genpb},
and show how they can be computed in a decomposed way, that is, unit by unit.

\begin{proposition}
  \label{prop:nodal:valuefunctionsbounds}
  For any admissible price vector
  $\price = \sequence{\price^\site}{\site \in \SITE} \in \CONE^\star$
  and for any admissible resource vector
  $\alloc =\sequence{\alloc^\site}{\site \in \SITE} \in -\CONE$,
  we have the following lower and upper decomposed estimates
  of the global minimum $V\opt$ of Problem~\eqref{eq:gen:genpb}:
  \begin{equation}
    \label{eq:nodal:generic:bounds}
    \sum_{\site \in \SITE} \underline V^\site\nc{\price^\site}
    \; \leq \; V\opt \; \leq \;
    \sum_{\site \in \SITE} \overline V^\site\nc{\alloc^\site}
    \eqfinp
  \end{equation}
\end{proposition}
\begin{proof}
  Because we have supposed that \( V^\site\nc{\price^\site} < +\infty \),
  the left hand side of Equation~\eqref{eq:nodal:generic:bounds} belongs
  to~$\ClosedIntervalOpen{-\infty}{+\infty}$. In the same way, the right
  hand side belongs to~$\OpenIntervalClosed{-\infty}{+\infty}$.
  For a given $\price = \sequence{\price^\site}{\site\in \SITE}  \in \CONE^\star$,
  we have
  \begin{align*}
    \sum_{\site \in \SITE} \underline V^\site\nc{\price^\site}
    &= \sum_{\site \in \SITE} \inf_{z^\site \in \cZ^\site} \;
      J^\site(z^\site) + \bscal{\price^\site}{\vartheta^\site(z^\site)}
      \eqfinv \\
    &= \inf_{z \in \cZ} \; {\sum_{\site \in \SITE} J^\site(z^\site)} +
      \bscal{\price}{\sequence{\vartheta^\site(z^\site)}{\site \in \SITE}}
      \eqfinv
      \tag{since \( z=\sequence{z^\site}{\site \in \SITE} \)}
    \\
    &\le \inf_{z \in \cZ} \; {\sum_{\site \in \SITE} J^\site(z^\site)} +
      \bscal{\price}{\sequence{\vartheta^\site(z^\site)}{\site \in \SITE}}
      \eqfinv \\
    &\hphantom{\le \inf_{z \in \cZ}} \st\
      {\sequence{\vartheta^\site(z^\site)}{\site \in \SITE}} \in - \CONE
      \tag{minimizing on a smaller set} \\
    &\le \inf_{z \in \cZ} \; {\sum_{\site \in \SITE} J^\site(z^\site) +0}
      \tag{as $\price\in \CONE^\star$ and by definition~\eqref{eq:nodal:dualcone} of~$\CONE^\star$}
    \\
    &\hphantom{\le \inf_{z \in \cZ}} \st\
      {\sequence{\vartheta^\site(z^\site)}{\site \in \SITE}} \in - \CONE
      \eqfinv
  \end{align*}
  which gives the lower bound inequality.

  The upper bound is easily obtained,
  as the optimal value $V\opt$ of Problem~\eqref{eq:gen:genpb} is given by
  $\inf_{\tilde\alloc \in -\CONE} \sum_{\site \in \SITE} \overline V^\site\nc{\tilde\alloc^\site}
  \leq \sum_{\site \in \SITE} \overline V^\site\nc{\alloc^\site}$ for any $\alloc \in -\CONE$.
\end{proof}

\subsection{The Special Case of Multistage Stochastic Optimization Problems
  \label{sec:nodal:globalpb}}

Now, we turn to the case where Problem~\eqref{eq:gen:genpb}
corresponds to a multistage stochastic optimization problem
elaborated from local data (local states, local controls, and
local noises), with global coupling constraints at each time step.
We use the notation \( \ic{r,s}=\na{r,r+1,\ldots,s-1,s} \) for two integers
$r \leq s$, and we consider a time span \( \ic{0,\final} \) where $\final
\in \NN^\star$ is a finite horizon.

\subsubsection{Local Data for Local Stochastic Control Problems}
\label{subsec:nodal:generic:localdata}

We detail the \emph{local} data describing
each unit. Let
$\ba{\XX_t^\site}_{t\in\ic{0,\final}}$,
$\ba{\UU_t^\site}_{t\in\ic{0,\final-1}}$
and $\ba{\WW_t^\site}_{t\in\ic{1,\final}}$
be sequences of measurable spaces for each unit $\site \in \SITE$.
We consider two other sequences of measurable vector spaces
$\ba{\ALLOCATION_t^\site}_{t\in\ic{0,\final-1}}$ and
$\ba{\PRICE_t^\site}_{t\in\ic{0,\final-1}}$
such that for all~$t$, $\ALLOCATION_t^\site$ and
$\PRICE_t^\site$ are paired spaces,
equipped with a bilinear form~$\pscal{\cdot}{\cdot}$.
We also introduce, for all~$\site \in \SITE$ and for
all~$t \in \ic{0, \final-1}$,
\begin{itemize}
\item
  measurable \emph{local dynamics}
  $g_t^\site : \XX_t^\site \times \UU_t^\site \times \WW_\post^\site \to \XX_\post^\site$,
\item
  measurable \emph{local coupling functions}
  $\Theta_t^\site: \XX_t^\site \times \UU_t^\site \to \ALLOCATION_t^\site$,
\item
  measurable \emph{local instantaneous costs}
  $L_t^\site: \XX_t^\site \times \UU_t^\site \times \WW_\post^\site \rightarrow
  \OpenIntervalClosed{-\infty}{+\infty}$,
\end{itemize}
and a measurable \emph{local final cost}
$K^\site: \XX^\site_\final \rightarrow \OpenIntervalClosed{-\infty}{+\infty}$.
We incorporate possible local constraints (for instance constraints
coupling the control with the state) directly in the instantaneous
costs~$L_t^\site$ and the final cost~$K^\site$, since they are extended real
valued functions which can possibly take the value $+\infty$.

From local data given above, we define the global state,
control, noise, resource and price spaces at time~$t$ as
\begin{equation*}
  \XX_t = \prod_{\site \in \SITE} \XX_t^\site , \;\:
  \UU_t = \prod_{\site \in \SITE} \UU_t^\site , \;\:
  \WW_t = \prod_{\site \in \SITE} \WW_t^\site , \;\:
  \ALLOCATION_t = \prod_{\site \in \SITE} \ALLOCATION_t^\site , \;\:
  \PRICE_t = \prod_{\site \in \SITE} \PRICE_t^\site
  \eqfinp
\end{equation*}
We suppose given a \emph{global constraint set}
$\CONE_t \subset \ALLOCATION_t$ \emph{at time~$t$}.
We define the global resource and price spaces~$\ALLOCATION$
and~$\PRICE$, and the global constraint set~$\CONE \subset \ALLOCATION$, as
\begin{equation}
  \ALLOCATION = \prod_{t=0}^{\final-1} \ALLOCATION_t
  \eqsepv
  \PRICE = \prod_{t=0}^{\final-1} \PRICE_t
  \eqsepv
  \CONE  = \prod_{t=0}^{\final-1} \CONE_t
  \subset \ALLOCATION
  \eqfinv
  \label{eq:nodal:global_constraint_set}
\end{equation}
and we denote by $\CONE^\star \subset \PRICE$ the dual cone of $\CONE$
(see Equation~\eqref{eq:nodal:dualcone}).

\subsubsection{Centralized and Decentralized Information Structures}
\label{subsec:nodal:generic:globaldata}

We introduce a probability space $(\Omega, \cF, \PP)$. For every unit
$\site \in\SITE$, we introduce \emph{local exogenous noise processes}
$\w^\site = \na{\w_t^\site}_{t\in \ic{1, \final}}$, where each
$\w_t^\site:\Omega\to\WW_t^\site$ is a random variable.\footnote{Random variables
  are denoted using bold letters.} We denote by
\begin{equation}
  \label{eq:nodal:generic:globalnoise}
  \w=(\w_1,\cdots,\w_\final)
  \quad\text{ where }\quad
  \w_t = \sequence{\w_t^\site}{\site \in \SITE}
\end{equation}
the \emph{global noise process}.
\begin{subequations}
  We consider two \emph{information structures}
  \cite[Chap.~3]{carpentier2015stochastic}:
  \begin{itemize}
  \item
    the \emph{centralized} information structure, represented
    by the filtration~$\cF = \np{\cF_t}_{t \in \ic{0,\final}}$,
    associated with the global noise process~$\w$ in~\eqref{eq:nodal:generic:globalnoise},
    where
    \begin{equation}
      \label{eq:nodal:globalinfo}
      \cF_t = \sigma(\w_1, \cdots, \w_t)
      = \sigma\bp{ \sequence{\w_1^\site}{\site \in \SITE}, \cdots,
        \sequence{\w_t^\site}{\site \in \SITE} }
    \end{equation}
    is the $\sigma$-field generated by all noises up to time~$t \in \ic{0,\final}$,
    with the convention~$\cF_0 = \{\emptyset,\Omega\}$,
  \item
    the \emph{decentralized} information structure, represented
    by the family \( \sequence{\cF^\site}{\site \in \SITE} \)
    of filtrations $\cF^\site = \np{\cF_t^\site}_{t \in \ic{0,\final}}$,
    where, for any unit $\site \in \SITE $ and any time $t \in \ic{0,\final}$,
    \begin{equation}
      \label{eq:nodal:localinfo}
      \cF_t^\site = \sigma(\w_1^\site, \cdots, \w_t^\site)
      \subset \cF_t = \bigvee_{\site' \in \SITE} \cF_t^{\site'}
      \eqfinv
    \end{equation}
    with~$\cF_0^\site = \{\emptyset,\Omega\}$.
    The \emph{local} $\sigma$-field~$\cF_t^\site$ captures
    the information provided by the uncertainties up to time~$t$,
    \emph{but only in unit~$\site$}.
  \end{itemize}
  \label{eq:nodal:infos}
\end{subequations}
In the sequel, for a given filtration~$\cG$ and a given measurable
space~$\YY$, we denote by $\LL^0(\Omega, \cG, \PP ; \YY)$ the space
of~\emph{$\cG$-adapted processes taking values in the space~$\YY$}.

\subsubsection{Global Stochastic Control Problem}

We denote by $\x_t = \sequence{\x_t^\site}{\site\in \SITE}$
and $\u_t = \sequence{\u_t^\site}{\site\in \SITE}$
families of random variables (each of them with values in~$\XX_t^\site$
and in $\UU_t^\site$).
The stochastic processes $\x = (\x_0,\cdots,\x_\final)$
and~$\u= (\u_0,\cdots,\u_{\final-1}) $ are called
\emph{global state} and \emph{global control} processes.
The stochastic processes $\x^\site = (\x_0^\site,\cdots,\x_\final^\site)$
and $\u^\site = (\u_0^\site,\cdots,\u_{\final-1}^\site)$ are called
\emph{local state} and \emph{local control processes}.

With the data detailed in \S\ref{subsec:nodal:generic:localdata}
and \S\ref{subsec:nodal:generic:globaldata}, we formulate
a family of optimization problems as follows.
At each time $t \in \ic{0, \final}$, the \emph{global value function}
$V_t : \prod_{\site \in \SITE} \XX_{t}^{\site} \rightarrow [-\infty,+\infty]$
is defined, for all
$\sequence{x_t^\site}{\site \in \SITE} \in \prod_{\site \in \SITE}
\XX_{t}^{\site}$, by
(with the convention~$V_\final=\sum_{\site \in \SITE}K^\site$)
\begin{subequations}
  \label{eq:nodal:vf}
  \begin{align}
    V_t\bp{\sequence{x_t^\site}{\site \in \SITE}} = \inf_{\x, \u} \;
    & \EE \bgc{\sum_{\site \in \SITE} \sum_{s=t}^{\final-1}
      L^\site_s(\va X_s^\site, \va U_s^\site, \va W^\site_{s+1}) +
      K^\site(\x_\horizon^\site)} \eqfinv
      \label{eq:nodal:expected_value}
    \\
    \st
    &\; \x_t^\site = x_t^\site \text{\, and \,} \FORALLTIMES{s}{t}{\final\!-\!1}
      \eqfinv \nonumber
    \\
    &\x_{s+1}^\site = \dynamic_s^\site(\x^\site_s, \u_s^\site, \w_{s+1}^\site)
      \eqsepv \x_t^\site = x_t^\site
      \eqfinv
      \label{eq:nodal:dynamic}
    \\
    &\sigma(\u_s^\site) \subset \cG_s^\site
      \eqfinv
      \label{eq:nodal:measurability}
    \\
    &
      \ba{\Theta_s^\site(\x_s^\site, \u_s^\site)}_{\site \in \SITE} \in -\CONE_s
      \eqfinp
      \label{eq:nodal:couplingcons}
  \end{align}
\end{subequations}
In the global value function~\eqref{eq:nodal:vf}, the expected
value is taken \wrt\ (with respect to) the global uncertainty
process~$(\w_\post, \cdots, \w_\final)$.
We assume that measurability and integrability assumptions hold true,
so that the expected value in~\eqref{eq:nodal:expected_value} is well
defined. Constraints~\eqref{eq:nodal:measurability} --- where
$\sigma(\u_s^\site)$ is the $\sigma$-field generated by the random
variable~$\u_s^\site$ --- express the fact that each decision
$\u_s^\site$ is $\cG_s^\site$-measurable, that is, measurable
either \wrt\ the global information~$\cF_s$ (centralized information
structure) available at time~$s$ (see~Equation~\eqref{eq:nodal:globalinfo})
or \wrt\ the local information $\cF_s^\site$ (decentralized
information structure) available at time~$s$ for unit~$\site$
(see~Equation~\eqref{eq:nodal:localinfo}), as detailed
in~\S\ref{subsec:nodal:generic:globaldata}.
Finally, Constraints~\eqref{eq:nodal:couplingcons} express
the global coupling constraint at time~$s$ between all units
and have to be understood in the $\PP$-almost sure sense.

We are mostly interested in the \emph{global optimization problem}~$
V_0\np{x_0}$,
where $x_0 = \sequence{x_0^\site}{\site \in \SITE} \in \XX_0$ is the initial
state, that is, Problem~\eqref{eq:nodal:vf} for~$t=0$.

\subsubsection{Local Price and Resource Value Functions}
\label{sec:nodal:localvaluefunctions}

As in \S\ref{subsec:nodal:bounds}, we define
local price and local resource value functions for the
global multistage stochastic optimization problems~\eqref{eq:nodal:vf}.

For this purpose, we introduce a duality pairing between stochastic processes.
For each $\site\in\SITE$, we consider subspaces
\( \widetilde{\LL}(\Omega,\cF,\PP ;\ALLOCATION^\site)
\subset \LL^0(\Omega, \cF, \PP ; \ALLOCATION^\site) \)
and
\( \widetilde{\LL}^{\star}(\Omega,\cF,\PP ;\PRICE^\site)
\subset \LL^0(\Omega, \cF, \PP ; \PRICE^\site) \) such that
 the duality product terms
 $\EE\bc{\sum_{t=0}^{\final-1}\pscal{\Price_t^\site}{\Theta_t^\site(\x_t^\site,\u_t^\site)}}$
 in Equation~\eqref{eq:nodal:priceproblem-t}
are well defined (like in the case of square integrable
random variables, when
$\Theta_t^\site(\x_t^\site, \u_t^\site) \in \LL^2(\Omega,\cF_t, \PP ; \RR^d)$
and $\Price_t^\site \in \LL^2(\Omega,\cF_t, \PP ; \RR^d)$).

Let $\site \in \SITE$ be a local unit, and
$\Price^\site = (\Price_0^\site, \cdots, \Price_{\final-1}^\site)
\in \widetilde{\LL}^{\star}(\Omega,\cF,\PP ;\PRICE^\site)$) be
a \emph{local price process} --- hence, adapted to the global
filtration $\cF$ in \eqref{eq:nodal:globalinfo} generated by
the global noises (note that we do not assume that it is adapted
to the local filtration $\cF^\site$ in \eqref{eq:nodal:localinfo}
generated by the local noises).
When specialized to the context of Problems~\eqref{eq:nodal:vf},
Equation~\eqref{eq:nodal:genericdualvf} gives,
at each time $t \in \ic{0, \final}$,
what we call \emph{local price value functions}
$\underline V^\site_t\nc{\Price^\site} : \XX_{t}^{\site} \rightarrow
\ClosedIntervalOpen{-\infty}{+\infty}$
defined, for all $x_t^\site \in \XX_t^\site$, by
(with the convention~$\underline V^\site_\final\nc{\Price^\site}=K^\site$)
\begin{align}
  \underline V^\site_t\nc{\Price^\site}(x_t^\site) = \inf_{\x^\site, \u^\site} \;
  & \EE \bigg[\sum_{s=t}^{\final-1}
    \Big(L^\site_s(\va X_s^\site,\va U_s^\site,\va W^\site_{s+1})
    \nonumber \\
  & \hspace{1.5cm} + \pscal{\Price_s^\site}{\Theta_s^\site(\x_s^\site, \u_s^\site)}\Big)
    +  K^\site(\x_\horizon^\site)\bigg] \eqfinv
    \label{eq:nodal:priceproblem-t}
  \\
  \st
  & \; \x_t^\site = x_t^\site \text{\, and \,} \FORALLTIMES{s}{t}{\final\!-\!1},
    \eqref{eq:nodal:dynamic}, \eqref{eq:nodal:measurability}.
    \nonumber
\end{align}
We suppose that
\( \underline V^\site_t\nc{\Price^\site}(x_t^\site) < +\infty \)
in~\eqref{eq:nodal:priceproblem-t}.
%
We define the \emph{global price value function}
$\underline V_t\nc{\Price^\site} : \XX_{t} \rightarrow \ClosedIntervalOpen{-\infty}{+\infty}$
at time~$t\in\ic{0,\final}$ as
the sum of the corresponding local price value functions, that is,
using the notation \( x_t=\sequence{x_t^\site}{\site \in \SITE} \),
\begin{equation}
  \label{eq:global:priceproblem}
  \underline V_t\nc{\Price}(x_t) = \sum_{\site \in \SITE}
  \underline V^\site_t\nc{\Price^\site}(x_t^\site)
  \eqsepv \forall x_t \in \XX_{t}
  \eqfinp
\end{equation}

In the same vein, let
$\Alloc^\site =(\Alloc_0^\site,\cdots,\Alloc_{\final-1}^\site)
\in \widetilde{\LL}(\Omega,\cF,\PP ; \ALLOCATION^\site)$
be a \emph{local resource process}.
Equation~\eqref{eq:nodal:genericprimalvf} gives,
at each time $t \in \ic{0, \final}$,
what we call \emph{local resource value function},
$\overline V^\site_t\nc{\Alloc^\site} : \XX_{t}^{\site} \rightarrow \OpenIntervalClosed{-\infty}{+\infty}$
defined, for all $x_t^\site \in \XX_t^\site$, by
(with the convention~$\qvaluefunc^\site_\final\nc{\Alloc^\site}=K^\site$)
\begin{subequations}
  \label{eq:nodal:quantproblem-t}
  \begin{align}
    & \qvaluefunc^\site_t\nc{\Alloc^\site}(x_t^\site) =
      \inf_{\x^\site, \u^\site} \;
      \EE \bgc{ \sum_{s=t}^{\final-1}
      L^\site_s(\va X_s^\site, \va U_s^\site, \va W^\site_{s+1}) +
      K^\site(\x_\horizon^\site)} \eqfinv \\
    \st
    & \, \x_t^\site = x_t^\site \text{\, and \,} \FORALLTIMES{s}{t}{\final\!-\!1},
      \eqref{eq:nodal:dynamic}, \eqref{eq:nodal:measurability}
      \text{\, and \,} \Theta_s^\site(\x_s^\site, \u_s^\site) = \Alloc_s^\site \eqfinp
  \end{align}
\end{subequations}
We suppose that
\( \qvaluefunc^\site_t\nc{\Alloc^\site}(x_t^\site) > -\infty \)
in~\eqref{eq:nodal:quantproblem-t}.
We define the \emph{global resource value function}
\( \overline V_t\nc{\Alloc} : \XX_{t} \rightarrow \OpenIntervalClosed{-\infty}{+\infty} \)
at time $t \in \ic{0, \final}$ as the sum
of the local resource value functions, that is,
\begin{equation}
  \label{eq:global:quantproblem}
  \overline V_t\nc{\Alloc}(x_t) = \sum_{\site \in \SITE}
  \overline V^\site_t\nc{\Alloc^\site}(x_t^\site)
  \eqsepv \forall x_t \in \XX_{t}
  \eqfinp
\end{equation}
We call the global processes
$\Price \in \widetilde{\LL}^{\star}(\Omega, \cF, \PP ; \PRICE)$
and
$\Alloc \in \widetilde{\LL}(\Omega,\cF,\PP ;\ALLOCATION)$
respectively
\emph{price coordination processes}
and
\emph{ressource coordination processes}.

\subsubsection{Global Upper and Lower Bounds}
\label{subsec:nodal:globalprocess}

Applying Proposition~\ref{prop:nodal:valuefunctionsbounds}
to the local price value functions~\eqref{eq:nodal:priceproblem-t}
and resource value functions~\eqref{eq:nodal:quantproblem-t}
makes it possible to bound the values of the global problems~\eqref{eq:nodal:vf}.
For this purpose, we first define the notion of \emph{admissible}
price and resource coordination processes.

\begin{subequations}
  We introduce the primal admissible set~$\CONESTO$ of stochastic processes
  associated with the almost sure constraints~\eqref{eq:nodal:couplingcons}:
  \begin{multline}
    \label{eq:nodal:primaladmissibleset}
    \CONESTO = \Big\{\va y = (\va y_0,\cdots,\va y_{\final-1})
    \in \widetilde{\LL}(\Omega,\cF,\PP;\ALLOCATION)
    \\
    \;\; \text{\st} \;\;
    \va y_t \in \CONE_t  \;\; \PP\text{-}\as \eqsepv
    \FORALLTIMES{t}{0}{\final\!-\!1}\Big\} \eqfinp
  \end{multline}
  Then, the dual admissible cone of~$\CONESTO$ is
  \begin{multline}
    \label{eq:nodal:dualadmissibleset}
    \CONESTO^\star = \Big\{\va z = (\va z_0, \cdots, \va z_{\final-1})
      \in \widetilde{\LL}^{\star}(\Omega,\cF,\PP;\PRICE) \\
      \text{\st} \;\;
      \EE \bc{\pscal{\va y_t}{\va z_t}} \geq 0 \eqsepv
      \forall \: \va y \in \CONESTO \eqsepv
      \FORALLTIMES{t}{0}{\final\!-\!1}\Big\}  \eqfinp
  \end{multline}
  \label{eq:nodal:admissiblesets}
\end{subequations}

We say that $\Price \in \widetilde{\LL}^{\star}(\Omega,\cF,\PP;\PRICE)$
is an \emph{admissible price coordination process}
if
$\Price\in\CONESTO^\star$, and that
$\Alloc \in \widetilde{\LL}(\Omega,\cF,\PP ;\ALLOCATION)$
is an \emph{admissible resource coordination process}
if $\Alloc \in -\CONESTO$.
By considering admissible coordination processes, we will now
bound up and down the global value functions~\eqref{eq:nodal:vf} with
the local value functions~\eqref{eq:nodal:priceproblem-t}
and~\eqref{eq:nodal:quantproblem-t}.

\begin{proposition}
  \label{prop:nodal:stochasticvaluefuncbounds}
  Let $\Price=\sequence{\Price^\site}{\site \in \SITE}\in\CONESTO^\star$
  be an admissible price coordination process,
  and let $\Alloc=\sequence{\Alloc^\site}{\site\in \SITE} \in -\CONESTO$
  be an admissible resource coordination process.
  Then, for all~$t\in\ic{0,\final}$ and
  for all $x_t = \sequence{x_t^\site}{\site \in \SITE} \in \XX_t$,
  we have the inequalities
  \begin{equation}
    \label{eq:nodal:stochasticvaluefuncbounds}
    \sum_{\site \in \SITE} \pvaluefunc_t^\site\nc{\Price^\site}(x_t^\site) \leq
    V_t(x_t) \leq
    \sum_{\site \in \SITE}  \qvaluefunc_t^\site\nc{\Alloc^\site}(x_t^\site) \eqfinp
  \end{equation}
\end{proposition}
\begin{proof}
  For~$t=0$, the proof of the following proposition is a direct
  application of Proposition~\ref{prop:nodal:valuefunctionsbounds}
  to Problem~\eqref{eq:nodal:vf}.
  For~$t \in \ic{1,\final\!-\!1}$,
  from the definitions~\eqref{eq:nodal:admissiblesets}
  of~$\CONESTO$ and~$\CONESTO^\star$, the assumption that
  $(\Alloc_0, \cdots, \Alloc_{\final-1})$
  (resp.~$(\Price_0, \cdots, \Price_{\final-1})$)
  is an admissible process implies that
  the reduced process $(\Alloc_t, \cdots, \Alloc_{\final-1})$
  (resp.~$(\Price_t, \cdots, \Price_{\final-1})$) is also admissible
  on the reduced time interval~$\ic{t,\final-1}$, hence the result
  by applying Proposition~\ref{prop:nodal:valuefunctionsbounds}.
\end{proof}

\section{Decomposition of Local Value Functions by Dynamic Programming}
\label{sec:genericdecomposition}

In~\S\ref{subsec:nodal:globalprocess}, we have obtained
upper and lower bounds of optimization problems by
spatial decomposition. We now give conditions under which
\emph{spatial decomposition} schemes can be made
\emph{compatible with temporal decomposition},
thus yielding a mix of spatial and temporal decompositions.

In~\S\ref{subsec:nodal:decomposedDPdeterministic}, we show
that the local price value functions~\eqref{eq:nodal:priceproblem-t}
and the local resource value functions~\eqref{eq:nodal:quantproblem-t}
can be computed by Dynamic Programming,
when price and resource processes are deterministic.
In~\S\ref{subsec:nodal:processdesign}, we sketch how to obtain
tighter bounds by appropriately choosing the deterministic price
and resource processes.
In~\S\ref{subsec:nodal:admissiblepolicy}, we show
how to use local price and resource value functions as
surrogates for the global Bellman value functions,
and then produce global admissible policies.
In~\S\ref{subsec:nodal:decentralizedinformation}, we analyze
the case of a \emph{decentralized information structure}.

In the sequel, we make the following key assumption.
\begin{assumption}
  \label{hyp:independent}
  The global uncertainty process $\np{\w_1, \cdots, \w_\final}$
  in \eqref{eq:nodal:generic:globalnoise}
  consists of stagewise independent random variables.
\end{assumption}

In the case where $\cG_t^\site=\cF_t$ for all~$t\in\ic{0,\final}$
and all~$\site \in \SITE$ (centralized information structure in~\S\ref{subsec:nodal:generic:globaldata}),
under Assumption~\ref{hyp:independent},
the global value functions~\eqref{eq:nodal:vf} satisfy
the Dynamic Programming equations \cite{carpentier2015stochastic}
\begin{subequations}
  \label{eq:globaldp}
  \begin{align}
    V_\final(x_\final)
    & = \sum_{\site \in \SITE} K^{\site}(x^{\site}_\final)
      \quad \text{ and, for \( t=\final\!-\!1, \ldots, 0 \),}
    \\
    V_t(x_t)
    & = \inf_{u_t \in \UU_t} \EE
      \bgc{\sum_{\site \in \SITE} L_t^{\site}(x_t^{\site}, u_t^{\site}, \w^{\site}_\post) +
      V_\post\bp{\sequence{\x_\post^\site}{\site \in \SITE}}}
    \\
    & \hphantom{u_t \in \UU_t} \st\ \;
      \x_\post^{\site} = \dynamic_t^{\site}(x_t^{\site}, u_t^{\site}, \w^{\site}_\post) \eqfinv \\
    & \hphantom{u_t \in \UU_t \st\ } \;
      \ba{\Theta_t^\site(x_t^\site, u_t^\site)}_{\site \in \SITE} \in -\CONE_t  \eqfinp
  \end{align}
\end{subequations}
In the case where $\cG_t^{\site}=\cF_t^{\site}$ for all~$t\in\ic{0,\final}$
and all~$\site \in \SITE$ (decentralized information structure
in~\S\ref{subsec:nodal:generic:globaldata}), the common assumptions
under which the global value functions~\eqref{eq:nodal:vf} satisfy
Dynamic Programming equations are not met.

\subsection{Decomposed Value Functions by  Deterministic Coordination Processes}
\label{subsec:nodal:decomposedDPdeterministic}

We prove now that, for deterministic coordination processes,
the local problems~\eqref{eq:nodal:priceproblem-t} and
\eqref{eq:nodal:quantproblem-t} satisfy local Dynamic Programming
equations.

We first study the local price value function~\eqref{eq:nodal:priceproblem-t}.
\begin{proposition}
  \label{prop:nodal:dppriceconstant}
  Let $\price^{\site}= (\price_0^{\site}, \cdots, \price_{\final-1}^{\site})
  \in \PRICE^{\site}$ be a deterministic price process.
  Then, be it for the centralized or the decentralized information
  structure (see \S\ref{subsec:nodal:generic:globaldata}),
  the local price value functions~\eqref{eq:nodal:priceproblem-t}
  satisfy the following recursive Dynamic Programming equations
  \begin{subequations}
    \label{eq:localdp}
    \begin{align}
      \pvaluefunc_\final^{\site}\nc{\price^{\site}}(x_\final^{\site}) =
      & \; K^{\site}(x_\final^{\site})
        \quad \text{and, for \( t=\final\!-\!1, \ldots, 0 \),}
      \\
      \pvaluefunc_t^{\site}\nc{\price^{\site}}(x_t^{\site}) =
      & \inf_{u_t^\site \in \UU_t^{\site}} \;
        \EE \Big[ L_t^{\site}(x_t^{\site}, u_t^{\site}, \w_\post^{\site}) +
        \pscal{\price_t^{\site}}{\Theta_t^{\site}(x_t^{\site}, u_t^{\site})}
      \\
      & \hspace{3.5cm} + \pvaluefunc_\post^{\site}\nc{\price^{\site}}
        \bp{\dynamic_t^{\site}(x_t^{\site}, u_t^{\site},
        \w_\post^{\site})} \Big]
        \eqfinp
        \nonumber
    \end{align}
  \end{subequations}
\end{proposition}
\begin{proof}
  Let $\price^{\site} = (\price_0^{\site},\cdots,\price_{\final-1}^{\site})\in \PRICE^{\site}$
  be a deterministic price vector. Then, the price value
  function~\eqref{eq:nodal:priceproblem-t} has the following expression:
  \begin{align}
    \underline V^{\site}_0\nc{\price^{\site}}(x_0^{\site}) = \inf_{\x^{\site}, \u^{\site}}
    & \EE \bigg[\sum_{t=0}^{\final-1}
      L^{\site}_t(\va X_t^{\site}, \va U_t^{\site}, \va W^{\site}_{t+1}) \nonumber
    \\
    & \hspace{2.0cm} + \pscal{\price_t^{\site}}{\Theta_t^{\site}(\x_t^{\site},\u_t^{\site})}
      + K^{\site}(\x_\horizon^{\site}) \bigg]
      \eqfinv
      \label{eq:nodal:priceproblem-deter}
    \\
    \st
    & \;\x_0^{\site} = x_0^{\site}\text{\, and \,}
      \FORALLTIMES{s}{0}{\final\!-\!1},
      \eqref{eq:nodal:dynamic}, \eqref{eq:nodal:measurability} \nonumber
      \eqfinp
  \end{align}
  In the case where~$\cG_t^{\site}=\cF_t$, and as
  Assumption~\ref{hyp:independent} holds true, the optimal value
  of Problem~\eqref{eq:nodal:priceproblem-deter} can be obtained
  by the recursive Dynamic Programming equations~\eqref{eq:localdp}.

  Consider now the case~$\cG_t^{\site}=\cF_t^{\site}$. Since the local
  value function and local dynamics in~\eqref{eq:nodal:priceproblem-deter}
  only depend on the local noise process~$\w^{\site}$, there is no loss of
  optimality to replace the constraint $\sigma(\u_t^{\site}) \subset \cF_t$
  by $\sigma(\u_t^{\site}) \subset \cF_t^{\site}$.
  Moreover,  Assumption~\ref{hyp:independent} implies that the local
  uncertainty process $(\w_{1}^{\site},\dots,\w_{\final}^{\site})$ consists
  of stagewise independent random variables, so that the solution of
  Problem~\eqref{eq:nodal:priceproblem-deter} can be obtained
  by the recursive Dynamic Programming equations~\eqref{eq:localdp}
  when replacing the \emph{global} $\sigma$-field~$\cF_t$
  by the \emph{local} $\sigma$-field~$\cF_t^{\site}$
  (see Equation~\eqref{eq:nodal:infos}).
\end{proof}

A similar result holds true for the local resource value
functions~\eqref{eq:nodal:quantproblem-t} as stated now in Proposition~\ref{prop:nodal:dpquantconstant}
whose proof is left to the reader.
\begin{proposition}
  \label{prop:nodal:dpquantconstant}
  Let $\alloc^{\site}= (\alloc_0^{\site}, \cdots, \alloc_{\final-1}^{\site})
  \in \ALLOCATION^{\site}$ be a deterministic resource process.
  Then, be it for the centralized or the decentralized information structure
  in~\S\ref{subsec:nodal:generic:globaldata},
  the local resource value functions~\eqref{eq:nodal:quantproblem-t}
  satisfy the following recursive Dynamic Programming equations
  \begin{subequations}
    \begin{align}
      \qvaluefunc_\final^{\site}\nc{\alloc^{\site}}(x_\final^{\site}) =
      & K^{\site}(x_\final^{\site})
        \quad \text{and, for \( t=\final\!-\!1, \ldots, 0 \),}
      \\
      \qvaluefunc_t^{\site}\nc{\alloc^{\site}}(x_t^{\site}) =
      & \inf_{u_t^\site \in \UU_t^{\site}} \EE \Bc{L_t^{\site}(x_t^{\site}, u_t^{\site}, \w_\post^{\site})
        + \qvaluefunc_\post^{\site}\nc{\alloc^{\site}}
        \bp{\dynamic_t^{\site}(x_t^{\site}, u_t^{\site}, \w_\post^{\site})}}
        \eqfinv
        \nonumber
      \\
      & \st\ \;\; \Theta_t^{\site}(x_t^{\site}, u_t^{\site}) = \alloc_t^{\site}
        \eqfinp
    \end{align}
    \label{eq:nodal:localdpquant}
  \end{subequations}
\end{proposition}

\subsection{Computing Upper and Lower Bounds, and Decomposed Value Functions}
\label{subsec:nodal:processdesign}

In the context of a deterministic \emph{admissible} price coordination process
$\price^{\site} = (\price_0^{\site},\cdots,\price_{\final-1}^{\site})\in \CONE^\star$
and resource process
$\alloc^{\site}= (\alloc_0^{\site}, \cdots, \alloc_{\final-1}^{\site}) \in \CONE$,
where~$\CONE$ is defined in~\eqref{eq:nodal:global_constraint_set}, the double inequality
\eqref{eq:nodal:stochasticvaluefuncbounds}
in Proposition~\ref{prop:nodal:stochasticvaluefuncbounds} becomes
\begin{equation}
  \label{eq:nodal:boundsvaluefunctiondeterministic}
  \sum_{\site \in \SITE} \pvaluefunc_t^{\site}\nc{\price^{\site}}(x_t^{\site})
  \leq V_t(x_t)\leq
  \sum_{\site \in \SITE} \qvaluefunc_t^{\site}\nc{\alloc^{\site}}(x_t^{\site})
  \eqfinp
\end{equation}
\begin{itemize}
\item
  Both in the lower bound and the upper bound of~$V_t$
  in~\eqref{eq:nodal:boundsvaluefunctiondeterministic},
  the sum over units~$\site\in\SITE$ materializes the spatial decomposition
  for the computation of the bounds. For each of the bounds, this
  decomposition leads to independent optimization subproblems
  that can be processed in parallel.
\item
  For a given unit~$\site\in\SITE$,
  the computation of the local value functions~$\pvaluefunc_t^{\site}\nc{\price^{\site}}$
  and~$\qvaluefunc_t^{\site}\nc{\alloc^{\site}}$ for $t \in \ic{0,T}$
  can be performed by Dynamic Programming
  as stated in Propositions
  \ref{prop:nodal:dppriceconstant} and~\ref{prop:nodal:dpquantconstant}.
  The corresponding loop in backward time materializes the temporal
  decomposition, processed sequentially.
\end{itemize}

Now, we suppose given an initial state
$x_0 = \sequence{x_0^\site}{\site \in \SITE} \in \XX_0$
and we sketch how, by suitably choosing the admissible coordination processes,
we can improve
the upper and lower bounds~\eqref{eq:nodal:boundsvaluefunctiondeterministic}
for~$V_0\np{x_0}$, that is, the optimal value of Problem~\eqref{eq:nodal:vf} for~$t=0$.

By Propositions \ref{prop:nodal:stochasticvaluefuncbounds}
and~\ref{prop:nodal:dppriceconstant}, for any deterministic
$\price=(\price_0,\cdots,\price_{\final-1}) \in \CONE^\star$,
we have
\( \sum_{\site \in \SITE} \pvaluefunc_0^{\site}\nc{\price^{\site}}(x_0^{\site})
\; \leq \; V_0(x_0) \).
As a consequence, solving the following optimization problem
\begin{equation}
  \sup_{\price \in \CONE^\star} \sum_{\site \in \SITE}
  \underline V^{\site}_0\nc{\price^{\site}}(x_0^{\site})
  \label{eq:nodal:relaxedconstraintdual}
\end{equation}
gives the greatest possible lower bound in the class
of deterministic admissible price coordination processes.
We can maximize Problem~\eqref{eq:nodal:relaxedconstraintdual}
\wrt~$\price$ using a gradient-like ascent algorithm.
Updating $\price$ requires the computation of the gradient of
$\sum_{\site \in \SITE} \pvaluefunc_0^{\site}\nc{\price^{\site}}(x_0^{\site})$,
obtained when computing the price value functions.
The standard update formula corresponding to the gradient algorithm
(Uzawa algorithm) can be replaced by more sophisticated methods (Quasi-Newton).

By Propositions \ref{prop:nodal:stochasticvaluefuncbounds}
and~\ref{prop:nodal:dpquantconstant}, for any
deterministic $\alloc=(\alloc_0,\cdots,\alloc_{\final-1}) \in -\CONE$,
we have
\(  V_0\bp{\sequence{x_0^\site}{\site \in \SITE}} \; \leq \;
\sum_{\site \in \SITE} \qvaluefunc_0^{\site}\nc{\alloc^{\site}}(x_0^{\site})
\).
As a consequence, solving the following optimization problem
\begin{equation}
  \label{eq:nodal:overconstraint}
  \inf_{\alloc\in -\CONE} \sum_{\site \in \SITE}
  \qvaluefunc_0^{\site}\nc{\alloc^{\site}}(x_0^{\site})
\end{equation}
gives the lowest possible upper bound in the set
of deterministic admissible resource coordination processes.
Again, we can minimize Problem~\eqref{eq:nodal:overconstraint}
\wrt~$\alloc$ using a gradient-like algorithm.
Updating $\alloc$ requires the computation of the gradient of
$\sum_{\site \in \SITE} \qvaluefunc_0^{\site}\nc{\alloc^{\site}}(x_0^{\site})$, obtained
when computing the resource value functions.
Again, the standard update formula corresponding to the gradient
algorithm can be replaced by more sophisticated methods.

At the end of the procedure, we have obtained
a deterministic admissible price coordination process
$\price=(\price_0,\cdots,\price_{\final-1}) \in \CONE^\star$
and a deterministic admissible resource coordination process
$\alloc=(\alloc_0,\cdots,\alloc_{\final-1}) \in -\CONE$
such that~$V_0\np{x_0}$,
the optimal value of Problem~\eqref{eq:nodal:vf} for~$t=0$,
is tightly bounded above and below like
in~\eqref{eq:nodal:boundsvaluefunctiondeterministic}
for~$t=0$. We have also obtained the solutions
$\na{\pvaluefunc_t^{\site}\nc{\price}}_{t\in \ic{0, \final}}$
and
$\na{\qvaluefunc_t^{\site}\nc{\alloc}}_{t\in \ic{0, \final}}$
of the recursive Dynamic Programming Equations~\eqref{eq:localdp} and~\eqref{eq:nodal:localdpquant}
associated with these coordination processes.

\subsection{Devising Policies}
\label{subsec:nodal:admissiblepolicy}

Now that we have decomposed value functions,
we show how to devise policies.
By \emph{policy}, we mean a sequence
\( \feedback = \ba{\feedback_{t}}_{t\in \ic{0, \final-1}} \) where,
for any \( t\in\ic{0,\final{-}1} \),
each $\feedback_t$ is a \emph{state feedback}, that is,
a measurable mapping \( \feedback_t : \XX_t\to\UU_t \).

Here, we suppose that we have at our disposal pre-computed \emph{local}
value functions $\na{\underline V_t^{\site}}_{t\in \ic{0, \final}}$
and $\na{\overline V_t^{\site}}_{t\in \ic{0, \final}}$ solving
Equations~\eqref{eq:localdp} for the price value functions
and Equations~\eqref{eq:nodal:localdpquant} for the resource
value functions.
For instance, one could use the functions
$\na{\pvaluefunc_t^{\site}\nc{\price}}_{t\in \ic{0, \final}}$
and
$\na{\qvaluefunc_t^{\site}\nc{\alloc}}_{t\in \ic{0, \final}}$
obtained at the end of~\S\ref{subsec:nodal:processdesign}.

Using the sum of these local value functions
as a surrogate for a global Bellman value function,
we propose two \emph{global} policies as follows
(supposing that the $\argmin$ are not empty and that the resulting expressions
provide measurable mappings \cite{bertsekas-shreve:1996}):

\noindent
1) a \emph{global price policy}
\( \underline \feedback =
\ba{\underline\feedback_{t}}_{t\in \ic{0, \final-1}} \)
with, for any \( t\in\ic{0,\final{-}1} \), the feedback
$\underline \feedback_t:\XX_t\to\UU_t$ defined
for all $x_t = \sequence{x_t^\site}{\site \in \SITE} \in \XX_t$ by
\begin{align}
  \underline\feedback_t(x_t) \in \argmin_{\sequence{u_t^\site}{\site \in \SITE}}
  & \; \EE\bgc{\sum_{\site \in \SITE} L_t^{\site}(x_t^{\site}, u_t^{\site}, \w_\post^{\site}) +
    \underline V_\post^{\site}\bp{g_t^{\site}(x_t^{\site}, u_t^{\site},
    \w_\post^{\site})}} \eqsepv
    \nonumber \\
  \st\
  & \; \ba{\Theta_t^\site(x_t^\site, u_t^\site)}_{\site \in \SITE}
    \in -\CONE_t
    \eqfinv
    \label{eq:nodal:globalpricepolicy}
\end{align}
2) a \emph{global resource policy}
\( \overline \feedback =
\ba{\overline\feedback_{t}}_{t\in \ic{0, \final-1}} \)
with, for any \( t \in \ic{0,\final{-}1} \), the feedback
$\overline \feedback_t: \XX_t \to \UU_t$ defined
for all $x_t = \sequence{x_t^\site}{\site \in \SITE} \in \XX_t$ by
\begin{align}
  \overline\feedback_t(x_t) \in \argmin_{\sequence{u_t^\site}{\site \in \SITE}}
  & \; \EE\bgc{\sum_{\site \in \SITE} L_t^{\site}(x_t^{\site}, u_t^{\site}, \w_\post^{\site}) +
    \overline V_\post^{\site}\bp{g_t^{\site}(x_t^{\site}, u_t^{\site},
    \w_\post^{\site})}}
    \eqfinv
    \nonumber \\
  \st\
  & \; \ba{\Theta_t^\site(x_t^\site, u_t^\site)}_{\site \in \SITE}
    \in -\CONE_t
    \eqfinp
    \label{eq:nodal:globalresourcepolicy}
\end{align}
Given a policy \( \feedback = \ba{\feedback_{t}}_{t\in \ic{0, \final-1}} \)
and any time $t \in \ic{0, \final}$, the expected cost of policy
$\feedback$ starting from state~$x_t$ at time~$t$ is equal to
\begin{align}
  V_t^\feedback(x_t) =
  & \; \EE\bgc{\sum_{\site \in \SITE} \sum_{s=t}^{\final-1}
    L_s^{\site}(\x_s^{\site},\feedback_s^{\site}(\x_s),\w_\post^{\site}) +
    K^{\site}(\x_\final^{\site})}
    \eqfinv
    \label{eq:nodal:costpolicy}
  \\
  \st\
  & \FORALLTIMES{s}{t}{\final\!-\!1} \eqsepv 
      \x_{s+1}^{\site} = \dynamic_s^{\site}(\x^{\site}_s, \feedback_s^{\site}(\x_s), \w_{s+1}^{\site})
      \eqsepv \x_t^{\site} = x_t^{\site}
      \eqfinp
      \nonumber
\end{align}
We provide several bounds hereafter.

\begin{proposition}
  \label{prop:nodal:boundresourcepolicy}
  Let $t \in \ic{0, \final}$ and $x_t = \sequence{x_t^\site}{\site \in \SITE} \in \XX_t$
  be a given state. Then, we have
  \begin{subequations}
    \begin{align}
      \sum_{\site \in \SITE} \pvaluefunc_t^{\site}(x_t^{\site}) \leq V_t(x_t)
      & \leq
        V_t^{\overline \feedback}(x_t)
        \leq \sum_{\site \in \SITE} \qvaluefunc_t^{\site}(x_t^{\site})
        \eqfinv
        \label{eq:nodal:boundresourcepolicy_a}
      \\
      V_t(x_t)
      & \leq \inf
        \ba{V_t^{\underline \feedback}(x_t),V_t^{\overline \feedback}(x_t)}
        \eqfinp
    \end{align}
    \label{eq:nodal:boundresourcepolicy}
  \end{subequations}
\end{proposition}
\begin{proof}
  We prove the right hand side inequality
  in~\eqref{eq:nodal:boundresourcepolicy_a} by backward induction.
  At time $t = \final$, the result is straightforward as
  $\qvaluefunc_t^{\site} = K^{\site}$ for all $\site \in \SITE$.
  Let $t \in \ic{0, \final-1}$ such that the right hand side inequality
  in~\eqref{eq:nodal:boundresourcepolicy_a} holds true at time $t+1$.
  Then, for all $x_t \in \XX_t$, Equation \eqref{eq:nodal:costpolicy}
  can be rewritten
  \begin{equation*}
    V_t^{\overline \feedback}(x_t) =
    \EE\bgc{\sum_{\site \in \SITE}\bp{ L_t^{\site}(x_t^{\site}, \overline \feedback_t^{\site}(x_t),
        \w_\post^{\site})} + V_\post^{\overline \feedback}(\x_\post) } \eqfinv
  \end{equation*}
  Using the induction assumption, we deduce that
  \begin{align*}
    V_t^{\overline \feedback}(x_t)
    &\leq
      \EE\bgc{\sum_{\site \in \SITE} L_t^{\site}(x_t^{\site}, \overline \feedback_t^{\site}(x_t),
      \w_\post^{\site}) + \qvaluefunc_\post^{\site}(\x_\post^{\site}) } \eqfinp\\
    \intertext{From the very definition~\eqref{eq:nodal:globalresourcepolicy}
    of the global resource policy, $\overline{\feedback}$ , we obtain}
    V_t^{\overline \feedback}(x_t)
    & \leq \inf_{\sequence{u_t^\site}{\site \in \SITE}} \;
      \EE \bgc{\sum_{\site \in \SITE} L_t^{\site}(x_t^{\site}, u_t^{\site}, \w_\post^{\site}) +
      \qvaluefunc_\post^{\site}(\x_\post^{\site})} \eqfinv \\
    & \hspace{1.0cm} \st\ \;
      \ba{\Theta_t^\site(x_t^\site, u_t^\site)}_{\site \in \SITE} \in -\CONE_t \eqfinp
  \end{align*}
  Introducing a deterministic admissible resource
  process $\sequence{r_t^\site}{\site \in \SITE} \in -\CONE_t$
  and restraining the constraint with it reinforces
  the inequality, thus giving
  \begin{subequations}
    \label{eq:nodal:proof:temp1}
    \begin{align}
      V_t^{\overline \feedback}(x_t)
      & \leq \inf_{{\sequence{u_t^\site}{\site \in \SITE}}} \;
        \EE \bgc{\sum_{\site \in \SITE} L_t^{\site}(x_t^{\site}, u_t^{\site}, \w_\post^{\site}) +
        \qvaluefunc_\post^{\site}(\x_\post^{\site}) } \\
      & \hspace{1.0cm} \st\ \;
        \Theta_t^\site(x_t^\site, u_t^\site) = r_t^\site \eqsepv \forall {\site \in \SITE}
        \eqfinv
    \end{align}
  \end{subequations}
  so that
  \begin{equation*}
    V_t^{\overline \feedback}(x_t) \leq \sum_{\site \in \SITE}
    \Bp{\inf_{u_t^{\site}}
      \EE\bc{L_t^{\site}(x_t^{\site}, u_t^{\site}, \w_\post^{\site}) +
        \qvaluefunc_\post^{\site}(\x_\post^{\site})}
      \;\; \st \;\; \Theta_t^{\site}(x_t^{\site},u_t^{\site}) = r_t^{\site}}
  \end{equation*}
  as we do not have any coupling left in \eqref{eq:nodal:proof:temp1}.
  By Equation~\eqref{eq:nodal:localdpquant}, we deduce that
  \(   V_t^{\overline \feedback}(x_t)\leq \sum_{\site \in \SITE}
  \qvaluefunc_t^{\site}(x_t^{\site}) \),
  hence the result at time~$t$.

  Furthermore, for any admissible policy $\feedback$,
  we have $V_t(x_t) \leq V_t^{\feedback}(x_t)$ as the global Bellman
  function gives the minimal cost starting at any point $x_t \in \XX_t$.
  We therefore obtain all the other inequalities
  in~\eqref{eq:nodal:boundresourcepolicy}.
\end{proof}

\subsection{Analysis of the Decentralized Information Structure}
\label{subsec:nodal:decentralizedinformation}

An interesting consequence of Propositions
\ref{prop:nodal:dppriceconstant} and~\ref{prop:nodal:dpquantconstant}
is that the local price and resource value functions
$\pvaluefunc_t^{\site}\nc{\price^{\site}}$ in~\eqref{eq:global:priceproblem}
and~$\qvaluefunc_t^{\site}\nc{\alloc^{\site}}$
in~\eqref{eq:global:quantproblem}
remain the same when choosing either the centralized information
structure or the decentralized one  in~\S\ref{subsec:nodal:generic:globaldata}.
By contrast, the global value functions~$V_t$ in~\eqref{eq:nodal:vf}
depend on that choice. Let us denote by~$V^{\mathrm{C}}_t$
(resp. $V^{\mathrm{D}}_t$) the value functions~\eqref{eq:nodal:vf}
in the centralized (resp. decentralized) case
where \( \sigma(\u_s^\site) \subset \cF_s^\site \)
(resp. \( \sigma(\u_s^\site) \subset \cF_s \)). Since the admissible
set induced by the constraint~\eqref{eq:nodal:measurability}
in the centralized case is larger than the one in the decentralized
case (because $\cF_t^{\site} \subset \cF_t$ by \eqref{eq:nodal:localinfo}),
we deduce that the lower bound is tighter for the centralized problem,
and the upper bound tighter for the decentralized problem:
for all $x_t = \sequence{x_t^\site}{\site \in \SITE} \in \XX_t$,
\begin{equation}
  \label{eq:boundsCandD}
  \sum_{\site \in \SITE} \pvaluefunc_t^{\site}\nc{\price^{\site}}(x_t^{\site})
  \leq V^{\mathrm{C}}_t(x_t) 
  \leq V^{\mathrm{D}}_t(x_t) 
  \leq \sum_{\site \in \SITE} \qvaluefunc_t^{\site}\nc{\alloc^{\site}}(x_t^{\site})
  \eqfinp
\end{equation}

Now, we show that, in some specific cases (but often encountered in practical
applications), the best upper bound
in~\eqref{eq:boundsCandD} is equal to the optimal value~$V^{\mathrm{D}}_t(x_t)$ of the decentralized problem.

\begin{proposition}
  \label{prop:upperequalD}
  If, for all~$t \in \ic{0,\final-1}$, we have the equivalence
  \begin{equation}
    \label{eq:upperequalD-ass}
    \begin{split}
      \ba{\Theta_t^\site(\x_t^\site, \u_t^\site)}_{\site \in \SITE}
      \in -\CONE_t
      \iff \\
      \bp{\exists \sequence{\alloc_t^\site}{\site \in \SITE}\in {-}\CONE_t \eqsepv
        \Theta_t^{\site}(\x_t^{\site}, \u_t^{\site})=\alloc_t^{\site}
        \quad \forall \site \in \SITE}
      \eqfinv
    \end{split}
  \end{equation}
  then the optimal value
  $V^{\mathrm{D}}_0(x_0)$ of the decentralized problem ---
  that is,
  given by~\eqref{eq:nodal:vf} where \( \sigma(\u_s^\site) \subset \cF_s^\site \)
  in~\eqref{eq:nodal:measurability} --- satisfies
  \begin{equation}
    \label{eq:upperequalD-prop}
    V_0^{\mathrm{D}}(x_0) =
    \inf_{\alloc \in -\CONE} \;
    \sum_{\site \in \SITE} \qvaluefunc_0^{\site}\nc{\alloc^{\site}}(x_0^{\site}) \eqfinp
  \end{equation}
\end{proposition}
\begin{proof}
  Using Assumption~\eqref{eq:upperequalD-ass},
  Problem~\eqref{eq:nodal:vf} for~$t=0$
  can be written as
  \begin{align*}
    V_0^{\mathrm{D}}(x_0) =
    & \inf_{\alloc\in -\CONE}
      \Bgp{\sum_{\site \in \SITE} \inf_{\x^{\site}, \u^{\site}} \;
      \EE\bgc{ \sum_{t=0}^{\final-1}
      L^{\site}_t(\va X_t^{\site}, \va U_t^{\site}, \va W^{\site}_{t+1}) +
      K^{\site}(\x_\horizon^{\site})}} \eqfinv \\
    &  \st\ \: \x_0^{\site} = x_0^{\site} \text{\, and \,}
      \FORALLTIMES{s}{0}{\final\!-\!1},
      \eqref{eq:nodal:dynamic}, \eqref{eq:nodal:measurability},
      \Theta_s^{\site}(\x_s^{\site}, \u_s^{\site}) = \alloc_s^{\site}
      \eqfinv
      \nonumber\\
    =
    & \inf_{\alloc\in -\CONE} \;
      \sum_{\site \in \SITE} \qvaluefunc_0^{\site}\nc{\alloc^{\site}}(x_0^{\site}) \eqfinv
  \end{align*}
  the last equality arising from the definition
  of~$\qvaluefunc_0^{\site}\nc{\alloc^{\site}}$ in~\eqref{eq:nodal:quantproblem-t}
  for $t=0$.
\end{proof}

As an application of the previous Proposition~\ref{prop:upperequalD},
we consider the case
of a decentralized information structure with an additional
\emph{independence assumption in space} (whereas
Assumption~\ref{hyp:independent} is an independence assumption \emph{in time}).

\begin{corollary}
  We consider the case of a decentralized information structure
  with the following two additional assumptions:
  \begin{itemize}
  \item
    the random processes $\w^\site$, for $\site \in \SITE$, are independent,
  \item
    the coupling constraints~\eqref{eq:nodal:couplingcons}
    are of the form
    \( \sum_{\site \in \SITE}\Theta_t^{\site}(\x_t^{\site}, \u_t^{\site}) = 0 \).
  \end{itemize}
  Then, the assumptions of Proposition~\ref{prop:upperequalD}
  are satisfied, so that Equality~\eqref{eq:upperequalD-prop} holds true.
  \label{cor:upperequalD}
\end{corollary}

\begin{proof}
  From the dynamic constraint~\eqref{eq:nodal:dynamic} and from
  the measurability constraint~\eqref{eq:nodal:measurability},
  we have that each term~$\Theta_t^{\site}(\x_t^{\site}, \u_t^{\site})$ is
  $\cF_t^{\site}$-measurable in the decentralized information structure case.
  Since the random processes $\w^\site$, for $\site \in \SITE$, are independent,
  so are the $\sigma$-fields~$\cF_t^{\site}$, for $\site \in \SITE$, from which we
  deduce that the random variables
  $\Theta_t^{\site}(\x_t^{\site}, \u_t^{\site})$ are independent.
  Now, these random variables sum up to zero.
  But it is well-known that, if a sum of independent random variables
  is zero, then every random variable in the sum is constant (deterministic).
  Hence, each random variable $\Theta_t^{\site}(\x_t^{\site}, \u_t^{\site})$ is constant.
  By introducing their constant values~$\sequence{\alloc_t^\site}{\site \in \SITE}$,
  the constraints~\eqref{eq:nodal:couplingcons} are written equivalently
  $\Theta_t^{\site}(\x_t^{\site},\u_t^{\site}) - \alloc_t^{\site} = 0$,
  $\forall \: \site \in\SITE$,
  and $\sum_{\site\in\SITE} \alloc_t^{\site} = 0$.
  We conclude with Proposition~\ref{prop:upperequalD}.
\end{proof}

\begin{remark}
  \label{rem:decentralizedpolicy}
  In the case of a decentralized information structure~\eqref{eq:nodal:localinfo},
  it seems difficult to produce Bellman-based online policies.
  Indeed, neither the global price policy in~\eqref{eq:nodal:globalpricepolicy}
  nor the global resource policy in~\eqref{eq:nodal:globalresourcepolicy}
  are implementable since both policies require the knowledge of the global
  state $\sequence{x_t^\site}{\site \in \SITE}$ for each unit~$\site$, which is incompatible
  with the information constraint~\eqref{eq:nodal:localinfo}.
  Nevertheless, one can use the results given by resource decomposition
  to compute a local state feedback as follows.
  For a given deterministic admissible resource
  process~$\alloc \in -\CONE$, solving at time~$t$ and for each~$\site \in\SITE$
  the subproblem
  \begin{align*}
    \overline\feedback_t^{\site}(x_t^{\site}) \in \argmin_{u_t^{\site}}
    & \; \EE\Bc{ L_t^{\site}(x_t^{\site}, u_t^{\site}, \w_\post^{\site}) +
      \overline V_\post^{\site}\bp{ g_t^{\site}(x_t^{\site}, u_t^{\site}, \w_\post^{\site}) }} \eqfinv \\
    \st\
    & \Theta_t^{\site}(x_t^{\site}, u_t^{\site}) = \alloc_t^{\site}
  \end{align*}
  generates a local state feedback
  \( \overline\feedback_t^{\site} : \XX_t^{\site} \to\UU_t \) which is both
  compatible with the decentralized information
  structure~\eqref{eq:nodal:localinfo}
  and such that the policy \( \overline \feedback =
  \ba{\overline\feedback_{t}}_{t\in \ic{0, \final-1}} \) is admissible
  as it satisfies the global coupling
  constraint~\eqref{eq:nodal:couplingcons}
  between all units because $\alloc \in -\CONE$,
  where $\CONE$ is defined in~\eqref{eq:nodal:global_constraint_set}.

  By contrast, replicating this procedure with a deterministic admissible
  price process would produce a policy which would not satisfy the global
  coupling constraint~\eqref{eq:nodal:couplingcons}.
\end{remark}

\section{Application to Microgrids Optimal Management}
\label{chap:district:numerics}

We illustrate the effectiveness of the two decomposition schemes
introduced in Sect.~\ref{sec:genericdecomposition}
by presenting numerical results.
In~\S\ref{Description_of_the_problems}, we describe an application
in the optimal management of urban microgrids.
In~\S\ref{ssec:nodalalgorithms}, we detail how we implement
algorithms to obtain bounds and policies.
In~\S\ref{Numerical_results}, we
illustrate the performance of the decomposition methods
with numerical results.

\subsection{Description of the Problems}
\label{Description_of_the_problems}

The energy management problem
and the structure of the microgrids
come from case studies
provided by the urban Energy Transition Institute
Efficacity\footnote{Established in 2014 with the French government support,
  Efficacity aims to develop and implement innovative
  solutions to build and manage energy-efficient 
  cities.}.
For more details on microgrid modeling and on the formulation
of associated optimization problems, the reader is referred
to the PhD thesis~\cite{thesepacaud}.

We represent a district microgrid by a directed graph
$(\NODES, \ARCS)$, with $\NODES$ the set of nodes and $\ARCS$
the set of arcs. Each node of the graph corresponds to a building.
The buildings exchange energy through the edges of the graph,
hence coupling the different nodes of the graph by static
constraints (Kirchhoff law).
We manage the microgrids over a given day
in summer, with decisions taken every 15mn, so that $\final = 96$.

Each building has its own electrical and domestic hot water demand profiles,
and possibly its own solar panel production.
At node~$\node$, we consider a random variable~$\va\w_t^\node$,
with values in $\WW_t^\node=\RR^2$, representing the following couple
of uncertainties:
the local electricity demand minus the production of the solar panel;
the domestic hot water demand.
We also suppose given a corresponding finite probability distribution on the set~$\WW_t^\node$.

Each building is equipped with an electrical hot water tank;
some buildings have solar panels and some others have batteries.
We view batteries and electrical hot water tanks as energy stocks
so that, depending on the presence of battery inside the building,
we introduce a state~$\x_t^\node$ at node~$\node$ with dimension~2 or~1
(energy stored inside the water tank and energy stored in the battery),
and the same with the control $\u_t^\node$ at node~$\node$
(power used to heat the tank and power exchanged with the battery).
Each node of a graph is modelled as a local control system
in which the cost function corresponds to import electricity from the external
grid. Summing the costs and taking the expectation
(supposing that the $\np{\w_1, \cdots, \w_\final}$
are stagewise independent random variables),
we obtain a global optimization problem of the form~\eqref{eq:nodal:vf}.

We consider six different problems with growing sizes.
Table~\ref{tab:numeric:pbsize} displays the different dimensions
considered.
\begin{table}[!ht]
  \centering
  {\normalsize
    \begin{tabular}{|c|ccccc|}
      \hline
      Problem           & $\card{\NODES}$ & $\card{\ARCS}$ & $\mathbf{dim(\XX_t)}$ & $dim(\WW_t)$ & $supp(\w_t)$ \\
      \hline
      \hline
      \textrm{3-nodes}  & 3               & 3               & \textbf{4}            & 6            & $10^3$    \\
      \textrm{6-nodes}  & 6               & 7               & \textbf{8}            & 12           & $10^6$    \\
      \textrm{12-nodes} & 12              & 16              & \textbf{16}           & 24           & $10^{12}$ \\
      \textrm{24-nodes} & 24              & 33              & \textbf{32}           & 48           & $10^{24}$ \\
      \textrm{48-nodes} & 48              & 69              & \textbf{64}           & 96           & $10^{48}$ \\
      \hline
    \end{tabular}
  }
  \caption{Microgrid management problems with growing dimensions}
  \label{tab:numeric:pbsize}
\end{table}
As an example, the 12-nodes problem consists of twelve buildings;
four buildings are equipped with a battery, and four other
buildings are equipped with solar panels.
The devices are dispatched so that a building equipped with a solar
panel is connected to at least one building with a battery.

\subsection{Computing Bounds, Decomposed Value Functions and Devising Policies}
\label{ssec:nodalalgorithms}

We apply the two decomposition algorithms, introduced
in~\S\ref{subsec:nodal:processdesign} and
in~\S\ref{subsec:nodal:admissiblepolicy},
to each problem as described in Table~\ref{tab:numeric:pbsize}.
We will term \emph{Dual Approximate Dynamic Programming}
(DADP) the price decomposition algorithm
and \emph{Primal Approximate Dynamic Programming} (PADP)
the resource decomposition algorithm described in
\S\ref{subsec:nodal:processdesign} and
in~\S\ref{subsec:nodal:admissiblepolicy}.
We compare DADP and PADP
with the well-known Stochastic Dual Dynamic Programming (SDDP)
algorithm (see~\cite{girardeau2014convergence} and references
inside) applied to the global problem.
In this part, we suppose given an initial state
$x_0 = \sequence{x_0^\site}{\site \in \SITE} \in \XX_0$.

Regarding the SDDP algorithm, it is not implementable in
a straightforward manner since the cardinality of the global noise
support becomes huge with the number~$\card{\NODES}$ of nodes
(see Table~\ref{tab:numeric:pbsize}), so that the exact computation
of an expectation \wrt\ the global uncertainty
$\w_t= \sequence{\w_t^\node}{\node\in \NODES}$ is out of reach.
To overcome this issue, we have resampled the probability distribution
of the global noise~$\sequence{\w_t^\node}{\node\in \NODES}$ at
each time~$t$ by using the $k$-means clustering method
(see \cite{rujeerapaiboon2018scenario}).
Thanks to the convexity properties of the problem, the optimal quantization
yields a new optimization problem
whose optimal value is a lower bound for the optimal value
of the original problem (see \cite{lohndorfmodeling} for details).
Thus, the exact lower bound given by SDDP with resampling remains
a lower bound for the exact lower bound given by SDDP without resampling,
which itself is, by construction, a lower bound for the original problem.

Regarding DADP and PADP, we use a quasi-Newton algorithm
to perform the maximization \wrt\ $\price$ in~\eqref{eq:nodal:relaxedconstraintdual}
and the minimization \wrt\ $\alloc$ in~\eqref{eq:nodal:overconstraint}.
More precisely, the quasi-Newton algorithm is performed using Ipopt 3.12
(see~\citep{wachter2006implementation}). The algorithm stops either
when a stopping criterion is fulfilled or when no descent direction
is found.

Each algorithm (SDDP, DADP, PADP) returns a sequence
of global value functions indexed by time.
Indeed, SDDP produces approximate global value functions,
and, for DADP (resp. PADP), we sum the local price value functions
(resp. the local resource value functions) obtained as
solutions of the recursive Dynamic Programming
equations~\eqref{eq:localdp} (resp.~\eqref{eq:nodal:localdpquant}),
for the deterministic admissible price coordination process
$\price=(\price_0,\cdots,\price_{\final-1}) \in \CONE^\star$
(resp. the deterministic admissible resource coordination process
$\alloc=(\alloc_0,\cdots,\alloc_{\final-1}) \in -\CONE$)
obtained at the end of~\S\ref{subsec:nodal:processdesign}
for an initial state
$x_0 = \sequence{x_0^\site}{\site \in \SITE} \in \XX_0$.

As explained in~\S\ref{subsec:nodal:admissiblepolicy},
these global value functions yield policies.
Thus, we have three policies (SDDP, DADP, PADP)
that we can compare.
As the policies are admissible,
the three expected values of the associated costs are
\emph{upper bounds} of the optimal value of the global optimization problem.

\subsection{Numerical Results}
\label{Numerical_results}

We compare the three algorithms (SDDP, DADP, PADP)
regarding their execution time
in~\S\ref{Computation_of_the_Bellman_value_functions},
the quality of their theoretical bounds
in~\S\ref{Quality_of_the_theoretical_bounds},
and the performance of their policies in simulation
in~\S\ref{Policy_simulation_results}.

\subsubsection{CPU Execution Time}
\label{Computation_of_the_Bellman_value_functions}

Table~\ref{tab:district:numeric:optres} details CPU execution time
and number of iterations before reaching stopping criterion
for the three algorithms.
\begin{table}[!ht]
  \centering
  {\normalsize
    \begin{tabular}{|l|ccccc|}
      \hline
      Problem            & \textrm{3-nodes}  \hspace{-0.2cm}
      & \textrm{6-nodes}  \hspace{-0.2cm}
      & \textrm{12-nodes} \hspace{-0.2cm}
      & \textrm{24-nodes} \hspace{-0.2cm}
      & \textrm{48-nodes} \hspace{-0.2cm} \\
      \hline
      dim($\XX_t$)       & 4                & 8                & 16
      & 32               & 64     \\
      \hline
      \hline
      SDDP CPU time      & 1'               & 3'               & 10'
      & 79'              & 453'   \\
      SDDP iterations    & 30               & 100              & 180
      & 500              & 1500   \\
      \hline
      \hline
      DADP CPU time      & 6'               & 14'              & 29'
      & 41'              & 128'   \\
      DADP iterations    & 27               & 34               & 30
      & 19               & 29     \\
      \hline
      \hline
      PADP CPU time      & 3'               & 7'               & 22'
      & 49'              & 91'    \\
      PADP iterations    & 11               & 12               & 20
      & 19               & 20     \\
      \hline
    \end{tabular}
  }
  \caption{Comparison of CPU time and number of iterations for SDDP, DADP and PADP}
  \label{tab:district:numeric:optres}
\end{table}
For a small-scale problem like \textrm{3-nodes} (second column
of Table~\ref{tab:district:numeric:optres}), SDDP is faster
than DADP and PADP. However, for the 48-nodes problem (last
column of Table~\ref{tab:district:numeric:optres}),
\emph{DADP and PADP} are \emph{more than three times faster}
than SDDP.

Figure~\ref{fig:nodal:cputime} depicts how much CPU
time take the different algorithms with respect to the state dimension.
For this case study, we observe
that the \emph{CPU time grows almost linearly} \wrt\ the dimension of the state
for DADP and PADP, whereas it grows exponentially for SDDP.
Otherwise stated, decomposition methods scale better than SDDP
in terms of CPU time for large microgrids instances.
\begin{figure}[!ht]
  \centering
  \includegraphics[width=10.0cm]{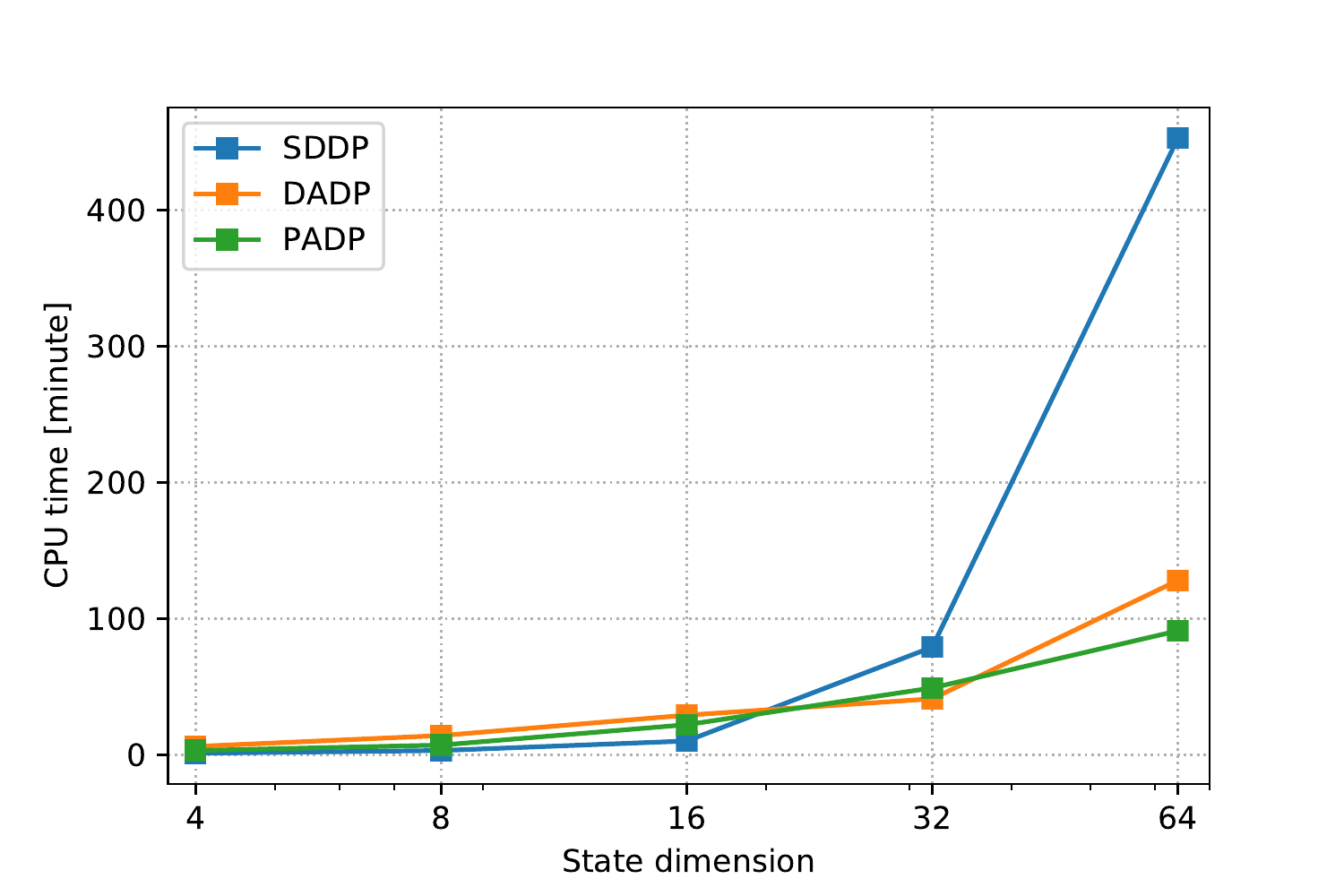} \\
  \caption{CPU time for the three algorithms as a function
    of the state dimension}
  \label{fig:nodal:cputime}
\end{figure}

\subsubsection{Quality of the Theoretical Bounds}
\label{Quality_of_the_theoretical_bounds}

In Table~\ref{tab:district:numeric:upperlower},
we give the lower and upper bounds (of the optimal
cost~$V_0(x_0)$ of the global optimization problem)
achieved by the three algorithms (SDDP, DADP, PADP).

We recall that SDDP returns a lower bound of the optimal
cost~$V_0(x_0)$, both by nature and also because we used
a suitable resampling of the global uncertainty distribution
instead of the original distribution itself (see the discussion
in~\S\ref{ssec:nodalalgorithms}). 
DADP and PADP lower and upper bounds are given by
Equation~\eqref{eq:nodal:relaxedconstraintdual}
and Equation~\eqref{eq:nodal:overconstraint} respectively.
In Table~\ref{tab:district:numeric:upperlower}, we observe that
\begin{itemize}
\item SDDP's and DADP's lower bounds are close to each other,
\item for problems with more than 12 nodes, DADP's lower
  bound is up to 2.6\% better than SDDP's lower bound,
\item the gap between PADP's upper bound and
  the two lower bounds is rather large.
\end{itemize}
\begin{table}[!ht]
  \centering
  {\normalsize
    \begin{tabular}{|l|ccccc|}
      \hline
      Problem        & \textrm{3-nodes} & \textrm{6-nodes} & \textrm{12-nodes} & \textrm{24-nodes} & \textrm{48-nodes} \\
      \hline
      \hline
      SDDP LB        & 225.2            & 455.9            & 889.7             & 1752.8            & 3310.3            \\
      \hline
      DADP LB        & 213.7            & 447.3            & 896.7             & 1787.0            & 3396.4            \\
      \hline
      PADP UB        & 252.1            & 528.5            & 1052.3            & 2100.7            & 4016.6            \\
      \hline
    \end{tabular}
  }
  \caption{Upper and lower bounds (of the optimal
    cost~$V_0(x_0)$ of the global optimization problem) given by SDDP, DADP and PADP}
  \label{tab:district:numeric:upperlower}
\end{table}

To sum up, DADP achieves a slightly better lower bound than SDDP,
with much less CPU time (and a parallel version of DADP would
give even better performance in terms of CPU time).

\subsubsection{Policy Simulation Performances}
\label{Policy_simulation_results}

In Table~\ref{tab:district:numeric:simulation},
we give the performances of the policies yielded by
the three algorithms.
The SDDP, DADP and PADP values are obtained by Monte Carlo simulation of the
corresponding policies on $5,000$ scenarios. The notation
$\pm$ corresponds to the 95\% confidence interval for the
numerical evaluation of the expected costs. We use
the value obtained by the SDDP policy as a reference,
a positive gap meaning that the corresponding policy
makes better than the SDDP policy.
All these values are \emph{statistical} upper bounds of the optimal
cost~$V_0(x_0)$ of the global optimization problem.

\begin{table}[H]
  \centering
  \resizebox{\textwidth}{!}{
    \begin{tabular}{|l|ccccc|}
      \hline
      Network      & \textrm{3-nodes} & \textrm{6-nodes} & \textrm{12-nodes} &
                                                                               \textrm{24-nodes} & \textrm{48-nodes} \\
      \hline
      \hline
      SDDP value   & 226 $\pm$ 0.6    & 471 $\pm$ 0.8    & 936 $\pm$ 1.1     & 1859 $\pm$ 1.6    & 3550 $\pm$ 2.3    \\
      \hline
      \hline
      DADP value   & 228 $\pm$ 0.6    & 464 $\pm$ 0.8    & 923 $\pm$ 1.2     & 1839 $\pm$ 1.6    & 3490 $\pm$ 2.3    \\
      Gap          & - 0.8 \%         & + 1.5 \%         & +1.4\%            & +1.1\%            & +1.7\%    \\
      \hline
      \hline
      PADP value   & 229 $\pm$ 0.6    & 471 $\pm$ 0.8    & 931 $\pm$ 1.1     & 1856 $\pm$ 1.6    & 3508  $\pm$ 2.2   \\
      Gap          & -1.3\%           & 0.0\%            & +0.5\%            & +0.2\%            & +1.2\%    \\
      \hline
    \end{tabular}
  }
  \caption{Simulation costs (Monte Carlo) for policies induced by
    SDDP, DADP and PADP}
  \label{tab:district:numeric:simulation}
\end{table}

We make the following observations:
\begin{itemize}
\item
  for problems with more than 6 nodes,
  both the DADP policy and the PADP policy beat
  the SDDP policy,
\item
  the DADP policy gives better results than the PADP policy,
\item
  comparing with the last line of Table
  \ref{tab:district:numeric:upperlower}, the statistical
  upper bounds 
  are much closer to SDDP and DADP lower bounds than PADP's
  exact upper bound.
\end{itemize}
For this last observation, our interpretation is as follows:
the PADP algorithm is penalized because, as the resource coordination
process is deterministic, it imposes constant
importation flows for every possible realization of
the uncertainties (see also the interpretation of PADP
in the case of a decentralized information structure
in~\S\ref{subsec:nodal:decentralizedinformation}).

\section{Conclusions}

We have considered multistage stochastic optimization problems
involving multiple units coupled by spatial static constraints.
We have presented a formalism for joint
temporal and spatial decomposition.
We have provided two fully parallelizable algorithms
that yield theoretical bounds, value functions
and admissible policies.
We have stressed the key role played by information structures in the
performance of the decomposition schemes.
We have tested these algorithms on the management of
several district microgrids. Numerical results have showed the effectiveness
of the approach: the price decomposition algorithm beats
the reference SDDP algorithm for large-scale problems with
more than 12~nodes, both in terms of theoretical bounds and
policy performance, and in terms of computation time. On problems
with up to 48~nodes (corresponding to 64~state variables), we have
observed that their performance scales well as the dimension of the state
grew: SDDP is affected by the well-known curse of dimensionality,
whereas decomposition-based methods are not.

Possible extensions are the following.
In~\S\ref{subsec:nodal:processdesign} and
in~\S\ref{subsec:nodal:admissiblepolicy},
we have presented a serial version of the decomposition algorithms,
but we believe that leveraging their parallel nature could decrease
further their computation time.
In~\S\ref{subsec:nodal:decomposedDPdeterministic},
we have only considered deterministic price and resource
coordination processes. Using larger search sets
for the coordination variables, e.g. considering
Markovian coordination processes, would make it
possible to improve the performance of the algorithms
(see \cite[Chap.~7]{thesepacaud} for further details).
However, one would need to analyze how to obtain a good
trade-off between accuracy and numerical performance.


\begin{thebibliography}{10}

\bibitem{bacaud2001bundle}
L{\'e}onard Bacaud, Claude Lemar{\'e}chal, Arnaud Renaud, and Claudia
  Sagastiz{\'a}bal.
\newblock Bundle methods in stochastic optimal power management: A
  disaggregated approach using preconditioners.
\newblock {\em Computational Optimization and Applications}, 20(3):227--244,
  2001.

\bibitem{barty2010decomposition}
Kengy Barty, Pierre Carpentier, and Pierre Girardeau.
\newblock Decomposition of large-scale stochastic optimal control problems.
\newblock {\em RAIRO-Operations Research}, 44(3):167--183, 2010.

\bibitem{bellman57}
Richard Bellman.
\newblock {\em Dynamic Programming}.
\newblock Princeton University Press, New Jersey, 1957.

\bibitem{bertsekas-shreve:1996}
D.~P. Bertsekas and S.~E. Shreve.
\newblock {\em Stochastic Optimal Control: The Discrete-Time Case}.
\newblock Athena Scientific, Belmont, Massachusets, 1996.

\bibitem{bertsekas1995dynamic}
Dimitri~P. Bertsekas.
\newblock {\em Dynamic programming and optimal control}, volume~1.
\newblock Athena Scientific Belmont, MA, third edition, 2005.

\bibitem{carpentier2018stochastic}
P.~Carpentier, J.-Ph. Chancelier, V.~Lecl{\`e}re, and F.~Pacaud.
\newblock Stochastic decomposition applied to large-scale hydro valleys
  management.
\newblock {\em European Journal of Operational Research}, 270(3):1086--1098,
  2018.

\bibitem{carpentier2015stochastic}
Pierre Carpentier, Jean-Philippe Chancelier, Guy Cohen, and Michel De~Lara.
\newblock {\em Stochastic Multi-Stage Optimization}, volume~75.
\newblock Springer, 2015.

\bibitem{carpentier2017decomposition}
Pierre Carpentier and Guy Cohen.
\newblock {\em D{\'e}composition-coordination en optimisation d{\'e}terministe
  et stochastique}, volume~81.
\newblock Springer, 2017.

\bibitem{cohen80}
G.~Cohen.
\newblock Auxiliary {P}roblem {P}rinciple and decomposition of optimization
  problems.
\newblock {\em Journal of {O}ptimization {T}heory and {A}pplications},
  32(3):277--305, 11 1980.

\bibitem{girardeau2014convergence}
Pierre Girardeau, Vincent Lecl\`{e}re, and Andrew~B. Philpott.
\newblock On the convergence of decomposition methods for multistage stochastic
  convex programs.
\newblock {\em Mathematics of Operations Research}, 40(1):130--145, 2014.

\bibitem{kim2018algorithmic}
Kibaek Kim and Victor~M Zavala.
\newblock Algorithmic innovations and software for the dual decomposition
  method applied to stochastic mixed-integer programs.
\newblock {\em Mathematical Programming Computation}, pages 1--42, 2018.

\bibitem{lohndorfmodeling}
Nils L{\"o}hndorf and Alexander Shapiro.
\newblock Modeling time-dependent randomness in stochastic dual dynamic
  programming.
\newblock {\em European Journal of Operational Research}, 273(2):650--671,
  2019.

\bibitem{thesepacaud}
Fran\c{c}ois Pacaud.
\newblock {\em Decentralized Optimization Methods for Efficient Energy
  Management under Stochasticity}.
\newblock Th\`ese de doctorat, {U}niversit\'e {P}aris-{Est}, 2018.

\bibitem{rockafellar1974conjugate}
R~Tyrrell Rockafellar.
\newblock {\em Conjugate duality and optimization}, volume~16.
\newblock Siam, 1974.

\bibitem{rockafellar1991scenarios}
R~Tyrrell Rockafellar and Roger J-B Wets.
\newblock Scenarios and policy aggregation in optimization under uncertainty.
\newblock {\em Mathematics of operations research}, 16(1):119--147, 1991.

\bibitem{rujeerapaiboon2018scenario}
Napat Rujeerapaiboon, Kilian Schindler, Daniel Kuhn, and Wolfram Wiesemann.
\newblock Scenario reduction revisited: Fundamental limits and guarantees.
\newblock {\em Mathematical Programming}, pages 1--36, 2018.

\bibitem{ruszczynski1997decomposition}
Andrzej Ruszczy{\'n}ski.
\newblock Decomposition methods in stochastic programming.
\newblock {\em Mathematical programming}, 79(1):333--353, 1997.

\bibitem{shapiro2012final}
Alexander Shapiro, Wajdi Tekaya, Joari~P da~Costa, and Murilo~P Soares.
\newblock Final report for technical cooperation between {G}eorgia {I}nstitute
  of {T}echnology and {ONS} -- {O}perador {N}acional do {S}istema
  {E}l\'{e}trico.
\newblock {\em Georgia Tech ISyE Report}, 2012.

\bibitem{wachter2006implementation}
Andreas W{\"a}chter and Lorenz~T Biegler.
\newblock On the implementation of an interior-point filter line-search
  algorithm for large-scale nonlinear programming.
\newblock {\em Mathematical programming}, 106(1):25--57, 2006.

\bibitem{watson2011progressive}
Jean-Paul Watson and David~L Woodruff.
\newblock Progressive hedging innovations for a class of stochastic
  mixed-integer resource allocation problems.
\newblock {\em Computational Management Science}, 8(4):355--370, 2011.

\end{thebibliography}
\end{document}